\documentclass[preprint,12pt,authoryear]{elsarticle}

\usepackage{amsmath,amssymb,amsthm,mathtools}
\usepackage{hyperref}

\newtheorem{theorem}{Theorem}
\newtheorem{proposition}{Proposition}
\newtheorem{lemma}{Lemma}
\newtheorem{corollary}{Corollary}
\newtheorem{definition}{Definition}

\newtheorem{remark}{Remark}

\newcommand{\dd}{\,\mathrm{d}}
\newcommand{\HCM}{\mathrm{HCM}}

\journal{Statistics \& Probability Letters}

\makeatletter
\def\ps@pprintTitle{%
    \let\@oddhead\@empty
    \let\@evenhead\@empty
    \let\@oddfoot\@empty
    \let\@evenfoot\@empty}
\makeatother

\hypersetup{
    colorlinks=true,
    linkcolor=blue,
    citecolor=blue,
    urlcolor=blue
}

\begin{document}

\begin{frontmatter}

\title{A Bessel-zero obstruction to hyperbolic complete monotonicity of noncentral chi-square densities}

\author[addr1]{Domingos S. P. Salazar\corref{cor1}}
\address[addr1]{Unidade de Educa\c{c}\~ao a Dist\^ancia e Tecnologia,
Universidade Federal Rural de Pernambuco,
52171-900 Recife, Pernambuco, Brazil}
\cortext[cor1]{Corresponding author.}

\begin{abstract}
Baricz, Prabhu K, Singh and Vijesh asked for the optimal hyperbolically completely monotone (HCM) range of the noncentral chi-square density.  The problem was motivated by the gap between known infinite divisibility and the stronger generalized-gamma-convolution/HCM classification.  We prove that the HCM range is exactly the central line.  More generally, for \(a,b>0\) and \(\theta\ge0\),
\begin{equation}
p_{a,b,\theta}(x)
=
\frac{b^a e^{-\theta}}{\Gamma(a)}
x^{a-1}e^{-bx}\,{}_0F_1(;a;b\theta x),
\qquad x>0,
\end{equation}
is HCM if and only if \(\theta=0\).  Thus the noncentral chi-square density satisfies
\begin{equation}
\chi_{\mu,\lambda}\in\HCM
\quad\Longleftrightarrow\quad
\lambda=0.
\end{equation}
The proof uses the leading small-\(u\) HCM signs of \(p(uv)p(u/v)\).  These signs are governed by complete Bell polynomials whose signed generating function is
\begin{equation}
e^{bt}{}_0F_1(;a;-b\theta t).
\end{equation}
A positive zero inherited from \(J_{a-1}\) rules out nonnegative Taylor coefficients when \(\theta>0\).  Consequently, Poisson shape-mixtures of HCM gamma densities need not be HCM.
\end{abstract}


\end{frontmatter}

\begingroup
\small
\noindent\textit{Pudim AI disclosure.}
This manuscript was developed with assistance from the Pudim AI research workflow.
Public provenance and APP ledger:
\href{https://github.com/pudim-project/pudim-ai-demo-zetalaw}{Pudim zeta-law repository}, APP-0047.
\par
\endgroup

\medskip

\section{Introduction}

A positive function \(f\) on \((0,\infty)\) is hyperbolically completely monotone (HCM) if for every \(u>0\), the product
\begin{equation}
v\longmapsto f(uv)f(u/v)
\end{equation}
is completely monotone as a function of the hyperbolic variable
\begin{equation}
w=v+v^{-1}.
\end{equation}
Namely, for each fixed \(u>0\), the resulting function \(g_u\) of \(w\) has derivatives of all orders on \((2,\infty)\) and
\begin{equation}
(-1)^n g_u^{(n)}(w)\ge0
\qquad(n\ge0,\ w>2).
\end{equation}
The class was developed systematically by Bondesson in connection with Thorin's generalized gamma convolutions (GGCs) and infinite divisibility; see \cite{Thorin1977,Bondesson1981,Bondesson1992}.  We use only the density-side definition above, but the motivation is the standard hierarchy: an HCM density yields a GGC law, and GGC laws are infinitely divisible \cite{Bondesson1992}.  Gamma densities are elementary HCM examples.  Indeed, if
\begin{equation}
\gamma_{a,b}(x)
=
\frac{b^a}{\Gamma(a)}x^{a-1}e^{-bx},
\qquad a,b>0,
\end{equation}
then
\begin{equation}
\gamma_{a,b}(uv)\gamma_{a,b}(u/v)
=
C_u e^{-bu(v+v^{-1})},
\end{equation}
which is completely monotone in \(w=v+v^{-1}\).

The noncentral chi-square density is a natural deformation of this central gamma case.  For degrees of freedom \(\mu>0\) and noncentrality \(\lambda>0\), the density is
\begin{equation}
\chi_{\mu,\lambda}(x)
=
\frac12 e^{-(x+\lambda)/2}
\left(\frac{x}{\lambda}\right)^{\mu/4-1/2}
I_{\mu/2-1}(\sqrt{\lambda x}),
\qquad x>0,
\end{equation}
with the usual limiting interpretation at \(\lambda=0\), where the central density is \(\gamma_{\mu/2,1/2}\).  The same family is tied to the generalized Marcum-\(Q\) function and to standard applications of the noncentral chi-square distribution; see, for example, \cite{BariczJankovMasirevicPogany2021}.  Equivalently, it is a Poisson mixture of central chi-square densities:
\begin{equation}
\chi_{\mu,\lambda}(x)
=
e^{-\lambda/2}
\sum_{m=0}^{\infty}
\frac{(\lambda/2)^m}{m!}\,
\chi_{\mu+2m,0}(x).
\end{equation}
Since each central chi-square density is a gamma density and hence HCM, this representation gives a natural test case for whether fixed-rate Poisson mixing over the gamma shape preserves HCM.

This test case is also a concrete open range problem in the recent literature on infinitely divisible modified Bessel distributions.  The noncentral chi-square law is already known to be infinitely divisible for all positive degrees of freedom by Ismail and Kelker \cite[Theorem~1.6]{IsmailKelker1979}; it also belongs to Bondesson's larger class of generalized convolutions of mixtures of exponential distributions \cite[p.~43]{Bondesson1981}, as recalled in \cite[Section~4.2]{BariczPrabhuSinghVijesh2026}.  Because HCM densities give GGC laws, and because neighboring modified-Bessel laws such as the \(K\)-distribution fall on the positive side of the HCM/GGC theory \cite[Section~2.5]{BariczPrabhuSinghVijesh2026}, Baricz, Prabhu K, Singh and Vijesh asked for the optimal range of parameters for which the noncentral chi-square density is HCM.  Their Theorem~25 establishes the first two necessary HCM signs: for \(\mu>1\), the hyperbolic product is decreasing when \(\lambda\le\mu\), and convex when \(2\lambda\le\mu\).  It therefore leaves open the full alternating derivative chain required for HCM \cite[Section~4.2 and Theorem~25]{BariczPrabhuSinghVijesh2026}.  Product-closure results such as \cite{Bondesson2015} explain why such positive examples are natural in the broader GGC/HCM closure story.

The main point of this paper is that the HCM range collapses to the central gamma line.  Any positive noncentrality destroys HCM.  More generally, every positive \({}_0F_1\)-gamma deformation of a gamma density fails to be HCM.  This is an HCM statement only; it makes no claim about the GGC range.

The mechanism is also the main conceptual contribution.  The failure is not detected by inspecting the density directly.  Instead, HCM imposes infinitely many alternating derivative signs on the hyperbolic product.  The leading small-\(u\) forms of these signs are governed by complete Bell polynomials.  Their signed generating function is a Bessel function in disguise, and a positive zero of \(J_{a-1}\) forces one of the required signs to fail.

\section{Main results}

We begin with the general \({}_0F_1\)-gamma deformation.  For \(a,b>0\) and \(\theta\ge0\), define
\begin{equation}
\label{eq:p-abtheta}
p_{a,b,\theta}(x)
=
\frac{b^a e^{-\theta}}{\Gamma(a)}
x^{a-1}e^{-bx}\,{}_0F_1(;a;b\theta x),
\qquad x>0.
\end{equation}
Here
\begin{equation}
{}_0F_1(;a;z)
=
\sum_{m=0}^{\infty}\frac{z^m}{m!(a)_m},
\qquad
(a)_m=\frac{\Gamma(a+m)}{\Gamma(a)}.
\end{equation}
The parameter \(\theta\) is the deformation parameter.  When \(\theta=0\), \eqref{eq:p-abtheta} is the gamma density \(\gamma_{a,b}\).  When \(\theta>0\), \eqref{eq:p-abtheta} is a positive \({}_0F_1\)-tilt of \(\gamma_{a,b}\).  The first structural observation is that \eqref{eq:p-abtheta} is not an artificial density.

\begin{proposition}[Poisson-gamma representation]
\label{prop:poisson-gamma}
For \(a,b>0\) and \(\theta\ge0\), \(p_{a,b,\theta}\) is a probability density on \((0,\infty)\), and
\begin{equation}
\label{eq:poisson-gamma}
p_{a,b,\theta}(x)
=
e^{-\theta}\sum_{m=0}^{\infty}
\frac{\theta^m}{m!}\,
\gamma_{a+m,b}(x),
\end{equation}
where
\begin{equation}
\gamma_{\alpha,b}(x)
=
\frac{b^\alpha}{\Gamma(\alpha)}x^{\alpha-1}e^{-bx}.
\end{equation}
Equivalently, if \(N\sim\mathrm{Poisson}(\theta)\), then \(p_{a,b,\theta}\) is the density of a gamma law with random shape \(a+N\) and fixed rate \(b\).
\end{proposition}

\begin{proof}
Using the defining series of \({}_0F_1\),
\begin{equation}
p_{a,b,\theta}(x)
=
\frac{b^a e^{-\theta}}{\Gamma(a)}
x^{a-1}e^{-bx}
\sum_{m=0}^{\infty}
\frac{(b\theta x)^m}{m!(a)_m}.
\end{equation}
Since \((a)_m=\Gamma(a+m)/\Gamma(a)\), this becomes
\begin{equation}
p_{a,b,\theta}(x)
=
e^{-\theta}
\sum_{m=0}^{\infty}
\frac{\theta^m}{m!}
\frac{b^{a+m}}{\Gamma(a+m)}x^{a+m-1}e^{-bx},
\end{equation}
which is \eqref{eq:poisson-gamma}.  Each summand is a gamma density, and the weights form a Poisson probability mass function.  Hence \(p_{a,b,\theta}\) is a probability density.
\end{proof}

The main theorem is the following sharp classification.

\begin{theorem}[Bessel-zero obstruction for \({}_0F_1\)-gamma deformations]
\label{thm:main}
Let \(a,b>0\) and \(\theta\ge0\).  Then
\begin{equation}
p_{a,b,\theta}\in\HCM
\quad\Longleftrightarrow\quad
\theta=0.
\end{equation}
Equivalently, every positive \({}_0F_1\)-gamma deformation of a gamma density fails to be hyperbolically completely monotone.
\end{theorem}

The proof is given in Section~\ref{sec:proof-main}, after the small-\(u\) Bell obstruction and Bessel-zero lemmas.

\begin{corollary}[Noncentral chi-square HCM range]
\label{cor:noncentral-chi-square}
Let \(\mu>0\) and \(\lambda\ge0\).  The noncentral chi-square density \(\chi_{\mu,\lambda}\) is hyperbolically completely monotone if and only if
\begin{equation}
\lambda=0.
\end{equation}
Thus the HCM range is exactly the central gamma line.
\end{corollary}

\begin{proof}
If \(\lambda=0\), then \(\chi_{\mu,0}=\gamma_{\mu/2,1/2}\), which is HCM by the gamma case in Theorem~\ref{thm:main}.  Suppose now that \(\lambda>0\).  Put
\begin{equation}
a=\frac{\mu}{2},
\qquad
b=\frac12,
\qquad
\theta=\frac{\lambda}{2}.
\end{equation}
Using the standard modified-Bessel representation of the noncentral chi-square density \cite{JohnsonKotzBalakrishnan1994} and the standard series identity for \(I_\nu\) \cite[Eq.~10.25.2]{DLMF},
\begin{equation}
I_{a-1}(z)
=
\frac{(z/2)^{a-1}}{\Gamma(a)}
{}_0F_1\!\left(;a;\frac{z^2}{4}\right),
\end{equation}
one obtains
\begin{equation}
\chi_{\mu,\lambda}(x)
=
\frac{2^{-a}e^{-\lambda/2}}{\Gamma(a)}
x^{a-1}e^{-x/2}
{}_0F_1\!\left(;a;\frac{\lambda x}{4}\right)
=
p_{a,1/2,\lambda/2}(x).
\end{equation}
By Theorem~\ref{thm:main}, this density is not HCM for \(\lambda>0\).  Together with the central case, this proves the claim.
\end{proof}

The same argument gives a more structural probabilistic consequence.

\begin{corollary}[Poisson shape-mixtures of gamma densities need not be HCM]
\label{cor:poisson-mixture-not-hcm}
The class of HCM densities is not closed under Poisson mixing over the gamma shape parameter, even when the rate is fixed.  More precisely, for every \(a,b,\theta>0\),
\begin{equation}
p_{a,b,\theta}(x)
=
e^{-\theta}\sum_{m=0}^{\infty}
\frac{\theta^m}{m!}\,
\gamma_{a+m,b}(x)
\end{equation}
is a Poisson mixture of HCM gamma densities, but \(p_{a,b,\theta}\notin\HCM\).
\end{corollary}

\begin{proof}
Each gamma density \(\gamma_{a+m,b}\) is HCM.  The identity is Proposition~\ref{prop:poisson-gamma}.  Since \(\theta>0\), Theorem~\ref{thm:main} gives \(p_{a,b,\theta}\notin\HCM\).
\end{proof}

\section{Hyperbolic complete monotonicity and the small-\texorpdfstring{\(u\)}{u} test}

We recall the two sign conventions used below.

\begin{definition}[Completely monotone functions]
A \(C^\infty\) function \(\phi\) on an interval is completely monotone if
\begin{equation}
(-1)^n\phi^{(n)}(x)\ge0
\end{equation}
for every \(n\ge0\) and every \(x\) in the interval.
\end{definition}

\begin{definition}[Hyperbolically completely monotone functions]
A positive function \(f\) on \((0,\infty)\) is hyperbolically completely monotone (HCM) if for every \(u>0\), the function
\begin{equation}
v\mapsto f(uv)f(u/v)
\end{equation}
is completely monotone as a function of
\begin{equation}
w=v+v^{-1}.
\end{equation}
Equivalently, for every \(u>0\), the function
\begin{equation}
H_u(w)=f(uv(w))f(u/v(w)),
\qquad
v(w)=\frac{w+\sqrt{w^2-4}}2,
\qquad w>2,
\end{equation}
satisfies
\begin{equation}
(-1)^n\frac{\dd^n}{\dd w^n}H_u(w)\ge0
\end{equation}
for all \(n\ge0\) and all \(w>2\).
\end{definition}

All complete-monotonicity assertions in the hyperbolic variable below are understood on the interval \(w\in(2,\infty)\).  For the density \(p_{a,b,\theta}\), the power \(x^{a-1}\) disappears from the hyperbolic product:
\begin{equation}
(uv)^{a-1}(u/v)^{a-1}=u^{2a-2}.
\end{equation}
Thus only the exponential factor and the \({}_0F_1\) factor can affect the HCM signs.

Let
\begin{equation}
F_a(z)={}_0F_1(;a;z).
\end{equation}
Since \(F_a(0)=1\), its logarithm is analytic near \(0\).  Write
\begin{equation}
\label{eq:log-Fa}
\log F_a(z)=\sum_{m=1}^{\infty}q_m z^m
\end{equation}
in a sufficiently small disk around the origin.  In particular,
\begin{equation}
q_1=\frac1a.
\end{equation}

For \(m\ge0\), define the hyperbolic polynomials
\begin{equation}
P_m(w)=v^m+v^{-m},
\qquad
w=v+v^{-1}.
\end{equation}
Then \(P_0(w)=2\), \(P_1(w)=w\), and
\begin{equation}
P_{m+1}(w)=wP_m(w)-P_{m-1}(w).
\end{equation}
Hence \(P_m\) is a polynomial of degree \(m\) with leading coefficient \(1\).

We shall use the complete exponential Bell polynomials \(\mathcal B_n\), defined by
\begin{equation}
\label{eq:bell-def}
\exp\left(\sum_{m=1}^{\infty}x_m\frac{z^m}{m!}\right)
=
\sum_{n=0}^{\infty}
\mathcal B_n(x_1,\dots,x_n)\frac{z^n}{n!},
\end{equation}
with \(\mathcal B_0=1\).

\begin{proposition}[Leading Bell obstruction]
\label{prop:bell-obstruction}
Let \(a,b,\theta>0\), and let \(p=p_{a,b,\theta}\).  Define
\begin{equation}
H_u(w)=p(uv(w))p(u/v(w)),
\qquad w>2.
\end{equation}
Let
\begin{equation}
d_1=b\left(\frac{\theta}{a}-1\right),
\qquad
d_m=m!q_m(b\theta)^m
\quad(m\ge2),
\end{equation}
where the \(q_m\)'s are defined by \eqref{eq:log-Fa}.  Then, for each fixed \(n\ge1\),
\begin{equation}
\label{eq:bell-leading}
\frac{1}{H_u(w)}
\frac{\dd^n}{\dd w^n}H_u(w)
=
u^n\mathcal B_n(d_1,\dots,d_n)
+
O(u^{n+1})
\qquad
(u\downarrow0),
\end{equation}
locally uniformly for \(w\) in compact subintervals of \((2,\infty)\).

Consequently, if \(p_{a,b,\theta}\) is HCM, then
\begin{equation}
\label{eq:hcm-necessary-bell}
(-1)^n\mathcal B_n(d_1,\dots,d_n)\ge0
\qquad(n\ge1).
\end{equation}
\end{proposition}

\begin{proof}
Fix a compact interval \(K\subset(2,\infty)\) and an integer \(n\ge1\).  On \(K\), the functions \(v(w)\) and \(v(w)^{-1}\) are bounded.  Since \(F_a(0)=1\), choose \(\rho>0\) such that \(F_a\) is zero-free on \(|z|<\rho\), and choose the branch of \(\log F_a\) there with value \(0\) at the origin.  For all sufficiently small \(u>0\), the two arguments \(b\theta uv(w)\) and \(b\theta u/v(w)\) lie in a compact subdisk of \(|z|<\rho\), uniformly for \(w\in K\).  Hence the Taylor series for \(\log F_a\), differentiated after composition with \(v(w)\) and \(v(w)^{-1}\) up to order \(n\), converges uniformly on \(K\).  In particular, all remainders below are locally uniform in \(w\).

For \(w>2\), let \(v=v(w)\).  Up to factors independent of \(w\),
\begin{equation}
H_u(w)
=
e^{-bu(v+v^{-1})}
F_a(b\theta uv)F_a(b\theta u/v).
\end{equation}
Taking logarithms and using \eqref{eq:log-Fa}, for sufficiently small \(u\) we have, uniformly on \(K\),
\begin{equation}
\log H_u(w)
=
C_u-buw+
\sum_{m=1}^{\infty}q_m(b\theta u)^mP_m(w),
\end{equation}
where \(C_u\) is independent of \(w\).

The analyticity just noted also controls the differentiated tail.  To see this explicitly, put
\begin{equation}
R_K=\sup_{w\in K}\max\{v(w),v(w)^{-1}\}.
\end{equation}
For each \(0\le r\le n\), differentiating \(P_m(w)=v(w)^m+v(w)^{-m}\) \(r\) times gives
\begin{equation}
\sup_{w\in K}|P_m^{(r)}(w)|
\le C_{K,r}m^rR_K^m
\qquad(m\ge1).
\end{equation}
Choose \(u_0>0\) so that \(|b\theta|u_0R_K\) lies inside the disk of convergence of \(\log F_a\).  Then, uniformly for \(0<u<u_0\) and \(w\in K\), the \(r\)-th derivative of the tail
\begin{equation}
\sum_{m\ge r+1}q_m(b\theta u)^mP_m(w)
\end{equation}
is bounded by
\begin{equation}
C_{K,r}\sum_{m\ge r+1}|q_m|\,m^r(|b\theta|u)^mR_K^m
=O(u^{r+1}).
\end{equation}
Thus the contribution of the terms \(m\ge r+1\) to the \(r\)-th \(w\)-derivative is \(O(u^{r+1})\), locally uniformly on \(K\).

Since \(P_m\) has degree \(m\) and leading coefficient \(1\),
\begin{equation}
\frac{\dd^r}{\dd w^r}P_m(w)=0
\quad(m<r),
\end{equation}
and
\begin{equation}
\frac{\dd^r}{\dd w^r}P_r(w)=r!.
\end{equation}
Thus, for \(r=1\),
\begin{equation}
\frac{\dd}{\dd w}\log H_u(w)
=
\left(q_1b\theta-b\right)u+O(u^2)
=
b\left(\frac{\theta}{a}-1\right)u+O(u^2)
=
d_1u+O(u^2),
\end{equation}
and for \(2\le r\le n\),
\begin{equation}
\frac{\dd^r}{\dd w^r}\log H_u(w)
=
r!q_r(b\theta)^r u^r+O(u^{r+1})
=
d_ru^r+O(u^{r+1}).
\end{equation}
The standard exponential derivative formula gives
\begin{equation}
\frac{H_u^{(n)}(w)}{H_u(w)}
=
\mathcal B_n
\left(
(\log H_u)'(w),
(\log H_u)''(w),
\dots,
(\log H_u)^{(n)}(w)
\right).
\end{equation}
Substituting the previous asymptotics into the weighted homogeneity of the complete Bell polynomials gives \eqref{eq:bell-leading}.  Explicitly, every monomial in \(\mathcal B_n\) has weighted degree \(n\), where the variable in position \(r\) has weight \(r\); replacing \((\log H_u)^{(r)}\) by \(d_ru^r+O(u^{r+1})\) therefore gives the leading term \(u^n\mathcal B_n(d_1,\dots,d_n)\) and an \(O(u^{n+1})\) remainder, locally uniformly on \(K\).

If \(p\) is HCM, then \((-1)^nH_u^{(n)}(w)\ge0\) for every \(u>0\), \(w>2\), and \(n\ge1\).  Since \(H_u(w)>0\), \((-1)^nH_u^{(n)}(w)/H_u(w)\ge0\).  Dividing \eqref{eq:bell-leading} by \(u^n>0\) and then letting \(u\downarrow0\), with any fixed \(w>2\), gives \eqref{eq:hcm-necessary-bell}.
\end{proof}

\section{The Bessel-zero obstruction}

The Bell-polynomial sequence in Proposition~\ref{prop:bell-obstruction} has an explicit generating function.

\begin{lemma}[Signed Bell generating function]
\label{lem:bell-generating}
With \(d_m\) as in Proposition~\ref{prop:bell-obstruction},
\begin{equation}
\label{eq:bell-gen-positive}
\sum_{n=0}^{\infty}
\mathcal B_n(d_1,\dots,d_n)\frac{z^n}{n!}
=
e^{-bz}F_a(b\theta z).
\end{equation}
Equivalently,
\begin{equation}
\label{eq:signed-bell-gen}
\sum_{n=0}^{\infty}
(-1)^n\mathcal B_n(d_1,\dots,d_n)\frac{t^n}{n!}
=
e^{bt}F_a(-b\theta t).
\end{equation}
Moreover, equivalently as an ordinary Taylor expansion,
\begin{equation}
\label{eq:laguerre-coeff}
e^{bt}F_a(-b\theta t)
=
\sum_{n=0}^{\infty}
\frac{b^n}{(a)_n}L_n^{(a-1)}(\theta)t^n,
\end{equation}
where \(L_n^{(a-1)}\) denotes the generalized Laguerre polynomial, with normalization
\begin{equation}
L_n^{(a-1)}(\theta)
=
\frac{(a)_n}{n!}\,{}_1F_1(-n;a;\theta).
\end{equation}
Consequently,
\begin{equation}
\label{eq:bell-laguerre-sign}
(-1)^n\mathcal B_n(d_1,\dots,d_n)
=
\frac{n!b^n}{(a)_n}L_n^{(a-1)}(\theta).
\end{equation}
\end{lemma}

\begin{proof}
By the defining identity \eqref{eq:bell-def}, as a formal power series and hence as an analytic identity near \(z=0\),
\begin{equation}
\sum_{n=0}^{\infty}
\mathcal B_n(d_1,\dots,d_n)\frac{z^n}{n!}
=
\exp\left(\sum_{m=1}^{\infty}d_m\frac{z^m}{m!}\right).
\end{equation}
Using the definitions of \(d_m\),
\begin{equation}
\sum_{m=1}^{\infty}d_m\frac{z^m}{m!}
=
-bz+\sum_{m=1}^{\infty}q_m(b\theta z)^m
=
-bz+\log F_a(b\theta z).
\end{equation}
Therefore the Taylor coefficients at the origin of the left-hand side are those of the entire function \(e^{-bz}F_a(b\theta z)\).  Since \(e^{-bz}F_a(b\theta z)\) is entire, the series in \eqref{eq:bell-gen-positive} has infinite radius of convergence and the identity holds for all \(z\in\mathbb C\).  Replacing \(z\) by \(-t\) gives \eqref{eq:signed-bell-gen}.

For \eqref{eq:laguerre-coeff}, expand directly, using the standard normalization of generalized Laguerre polynomials \cite[Eq.~18.5.12]{DLMF}:
\begin{equation}
e^{bt}F_a(-b\theta t)
=
\sum_{r=0}^{\infty}\frac{(bt)^r}{r!}
\sum_{k=0}^{\infty}
\frac{(-b\theta t)^k}{k!(a)_k}.
\end{equation}
The coefficient of \(t^n\) is
\begin{equation}
b^n\sum_{k=0}^{n}
\frac{(-\theta)^k}{(n-k)!k!(a)_k}
=
\frac{b^n}{n!}\,
{}_1F_1(-n;a;\theta).
\end{equation}
Using
\begin{equation}
L_n^{(a-1)}(\theta)
=
\frac{(a)_n}{n!}\,{}_1F_1(-n;a;\theta),
\end{equation}
we obtain \eqref{eq:laguerre-coeff}.
\end{proof}

The obstruction is now a one-line consequence of a classical Bessel identity.

\begin{lemma}[Positive zero of the signed generator]
\label{lem:positive-zero}
Let \(a,b,\theta>0\).  Then the entire function
\begin{equation}
t\mapsto e^{bt}{}_0F_1(;a;-b\theta t)
\end{equation}
has a positive real zero.
\end{lemma}

\begin{proof}
The classical identity \cite[Eq.~10.16.9]{DLMF}
\begin{equation}
{}_0F_1\!\left(;a;-\frac{x^2}{4}\right)
=
\Gamma(a)\left(\frac{x}{2}\right)^{1-a}J_{a-1}(x)
\end{equation}
holds for \(a>0\).  The prefactor \(\Gamma(a)(x/2)^{1-a}\) is nonzero for \(x>0\).  Since \(a-1>-1\) is real, the Bessel function \(J_{a-1}\) has positive real zeros \cite[Section~10.21(i)]{DLMF}.  Let \(j_{a-1,1}>0\) denote its first positive zero.  Set
\begin{equation}
t_0=\frac{j_{a-1,1}^2}{4b\theta}>0.
\end{equation}
Then
\begin{equation}
{}_0F_1(;a;-b\theta t_0)
=
{}_0F_1\!\left(;a;-\frac{j_{a-1,1}^2}{4}\right)
=
0.
\end{equation}
Multiplication by \(e^{bt_0}>0\) does not change the zero.
\end{proof}

\section{Proof of Theorem~\ref{thm:main}}
\label{sec:proof-main}

First suppose \(\theta=0\).  Then
\begin{equation}
p_{a,b,0}(x)=\gamma_{a,b}(x)
=
\frac{b^a}{\Gamma(a)}x^{a-1}e^{-bx}.
\end{equation}
For \(u>0\),
\begin{equation}
\gamma_{a,b}(uv)\gamma_{a,b}(u/v)
=
C_u e^{-bu(v+v^{-1})}
=
C_u e^{-buw}.
\end{equation}
Therefore
\begin{equation}
(-1)^n\frac{\dd^n}{\dd w^n}
\left(C_u e^{-buw}\right)
=
C_u(bu)^ne^{-buw}\ge0,
\end{equation}
so \(p_{a,b,0}\in\HCM\).

Now suppose \(\theta>0\), and assume for contradiction that \(p_{a,b,\theta}\in\HCM\).  Proposition~\ref{prop:bell-obstruction} gives
\begin{equation}
(-1)^n\mathcal B_n(d_1,\dots,d_n)\ge0
\qquad(n\ge1).
\end{equation}
Since \(\mathcal B_0=1\), all Taylor coefficients of
\begin{equation}
e^{bt}F_a(-b\theta t)
=
\sum_{n=0}^{\infty}
(-1)^n\mathcal B_n(d_1,\dots,d_n)\frac{t^n}{n!}
\end{equation}
are nonnegative, and the constant coefficient is \(1\).  By Lemma~\ref{lem:bell-generating}, this is an entire Taylor expansion, so it may be evaluated at every \(t>0\).  Hence this entire function must be strictly positive for every \(t>0\).  Indeed, at every \(t>0\),
\begin{equation}
\sum_{n=0}^{\infty}
(-1)^n\mathcal B_n(d_1,\dots,d_n)\frac{t^n}{n!}
\ge1.
\end{equation}
This contradicts Lemma~\ref{lem:positive-zero}, which gives a positive zero.  Therefore \(p_{a,b,\theta}\notin\HCM\) whenever \(\theta>0\).  The proof is complete.
\(\square\)

\begin{remark}[Laguerre-polynomial form of the obstruction]
By \eqref{eq:bell-laguerre-sign}, the HCM necessary signs would imply
\begin{equation}
L_n^{(a-1)}(\theta)\ge0
\qquad(n\ge0).
\end{equation}
The Bessel-zero argument proves that for every \(a>0\) and every \(\theta>0\), at least one index \(n\ge1\) satisfies
\begin{equation}
L_n^{(a-1)}(\theta)<0.
\end{equation}
For such an \(n\), Proposition~\ref{prop:bell-obstruction} gives an actual violation of the \(n\)-th HCM derivative sign for all sufficiently small \(u>0\).
\end{remark}

\begin{remark}[Why the result is not a mere convexity counterexample]
The obstruction is stronger than a low-order failure such as non-log-convexity or non-monotonicity.  The first few HCM signs may hold for some parameter ranges.  The proof shows that no matter how many early signs survive, the infinite HCM sign sequence cannot survive positive deformation, because the signed Bell generator has a positive Bessel zero.
\end{remark}

\section{Final remarks}

The theorem gives the sharp HCM answer to the noncentral chi-square range problem raised in \cite[Section~4.2]{BariczPrabhuSinghVijesh2026}: the central gamma line is the whole HCM range, and every positive noncentrality destroys HCM.  It also gives a concrete closure failure: fixed-rate Poisson mixing over the shape parameter can take gamma, hence HCM, densities outside the HCM class.  This negative result is deliberately limited to density-side HCM.  It does not contradict the known infinite divisibility of the noncentral chi-square law, and it does not decide whether any positive-noncentrality subfamily belongs to the generalized-gamma-convolution class.


\begin{thebibliography}{99}

\bibitem[Baricz et~al.(2026)]{BariczPrabhuSinghVijesh2026}
\'A. Baricz, D. Prabhu K, S. Singh, and A. Vijesh V.,
\newblock \emph{Infinitely divisible modified Bessel distributions},
\newblock Pacific Journal of Mathematics \textbf{343} (2026), 261--313.
\newblock arXiv:2406.17721.
\newblock DOI: \href{https://doi.org/10.2140/pjm.2026.343.261}{10.2140/pjm.2026.343.261}.

\bibitem[Baricz et~al.(2021)]{BariczJankovMasirevicPogany2021}
\'A. Baricz, D. Jankov Ma\v{s}irevi\'c, and T. K. Pog\'any,
\newblock \emph{Approximation of CDF of non-central chi-square distribution by mean-value theorems for integrals},
\newblock Mathematics \textbf{9} (2021), 129.
\newblock DOI: \href{https://doi.org/10.3390/math9020129}{10.3390/math9020129}.

\bibitem[Bondesson(1981)]{Bondesson1981}
L. Bondesson,
\newblock \emph{Classes of infinitely divisible distributions and densities},
\newblock Z. Wahrscheinlichkeitstheorie Verw. Gebiete \textbf{57} (1981), 39--71.
\newblock DOI: \href{https://doi.org/10.1007/BF00533713}{10.1007/BF00533713}.

\bibitem[Bondesson(1992)]{Bondesson1992}
L. Bondesson,
\newblock \emph{Generalized Gamma Convolutions and Related Classes of Distributions and Densities},
\newblock Lecture Notes in Statistics, vol. 76,
\newblock Springer, New York, 1992.

\bibitem[Bondesson(2015)]{Bondesson2015}
L. Bondesson,
\newblock \emph{A class of probability distributions that is closed with respect to addition as well as multiplication of independent random variables},
\newblock J. Theoret. Probab. \textbf{28} (2015), 1063--1081.
\newblock DOI: \href{https://doi.org/10.1007/s10959-013-0523-y}{10.1007/s10959-013-0523-y}.

\bibitem[NIST DLMF(2026)]{DLMF}
NIST Digital Library of Mathematical Functions,
\newblock F. W. J. Olver, A. B. Olde Daalhuis, D. W. Lozier, B. I. Schneider, R. F. Boisvert,
C. W. Clark, B. R. Miller, B. V. Saunders, H. S. Cohl, and M. A. McClain, eds.,
\newblock Version 1.2.7, release date 2026-06-15.
\newblock \url{https://dlmf.nist.gov/}.

\bibitem[Ismail and Kelker(1979)]{IsmailKelker1979}
M. E. H. Ismail and D. H. Kelker,
\newblock \emph{Special functions, Stieltjes transforms and infinite divisibility},
\newblock SIAM J. Math. Anal. \textbf{10} (1979), 884--901.
\newblock DOI: \href{https://doi.org/10.1137/0510083}{10.1137/0510083}.

\bibitem[Johnson et~al.(1994)]{JohnsonKotzBalakrishnan1994}
N. L. Johnson, S. Kotz, and N. Balakrishnan,
\newblock \emph{Continuous Univariate Distributions, Volume 1},
\newblock 2nd ed.,
\newblock Wiley, New York, 1994.

\bibitem[Thorin(1977)]{Thorin1977}
O. Thorin,
\newblock \emph{On the infinite divisibility of the Pareto distribution},
\newblock Scand. Actuar. J. \textbf{1977} (1977), 31--40.
\newblock DOI: \href{https://doi.org/10.1080/03461238.1977.10405623}{10.1080/03461238.1977.10405623}.

\end{thebibliography}
\end{document}